\RequirePackage[l2tabu, orthodox]{nag}
\documentclass[11pt]{amsart} %{{{1

\author{Layne Hall}
\address{School of Mathematical Sciences, Monash University, Victoria 3800 Australia}
\email{layne.hall@monash.edu}
\author{Andy Hammerlindl}
\address{School of Mathematical Sciences, Monash University, Victoria 3800 Australia}
\urladdr{ http://users.monash.edu.au/~ahammerl/} 
\email{andy.hammerlindl@monash.edu}
\title{Partially hyperbolic surface endomorphisms}
\providecommand{\keyword}[1]
{\textbf{Keywords:} #1}

\providecommand{\codes}[1]
{\textbf{MSC 2000:} #1}

\providecommand{\acknowledgement}
{{\noindent \footnotesize{\textbf{\textsc{Acknowledgements:}}}\quad{The authors thank Jonathan Bowden for helpful proofreading and discussion.}}}

\newcommand\funding[1]{%
	\begingroup
	\renewcommand\thefootnote{}\footnote{#1}%
	\addtocounter{footnote}{-1}%
	\endgroup
}

\usepackage{graphicx}			
\usepackage{amssymb}
\usepackage{amsmath}
\usepackage{amsthm}
\usepackage{xcolor}
\usepackage{mathtools}
\usepackage{tikz}
\usepackage{float}
\usepackage{bold-extra}
\usepackage{graphicx}
\usepackage[font={small}]{caption}
\usepackage{afterpage}
\usepackage{verbatim}
\usepackage{pgfplots}
\usepackage{csquotes}
\usepackage{bm}
\usepackage{enumerate}
\usepackage{textcomp}
\usepackage{titlesec}
\usepackage{nameref}
\usepackage{lipsum}
\usepackage{enumitem}
\usepackage{mathrsfs}
\usepackage[capitalise]{cleveref}
\usepackage[toc,page]{appendix}
\theoremstyle{definition}

\theoremstyle{Theorem}
\newtheorem{thm}{Theorem}

\newtheorem{theorem}{Theorem}

\newtheorem{prop}[thm]{Proposition}
\theoremstyle{Theorem}
\newtheorem{lem}[thm]{Lemma}
\theoremstyle{Theorem}

\theoremstyle{Theorem}
\theoremstyle{Theorem}
\newtheorem{cor}[thm]{Corollary}
\theoremstyle{remark}

\DeclareMathOperator{\eps}{\varepsilon}

\DeclareMathOperator{\bbT}{\mathbb{T}}
\DeclareMathOperator{\bbR}{\mathbb{R}}

\DeclareMathOperator{\bbZ}{\mathbb{Z}}
\DeclareMathOperator{\Cone}{\mathcal{C}}

\DeclareMathOperator{\bbN}{\mathbb{N}}
\DeclareMathOperator{\bbS}{\mathbb{S}}
\def\restrict#1{\raise-.5exhbox{\ensuremath|}_{#1}}

\numberwithin{thm}{section}
\numberwithin{equation}{section}
\numberwithin{figure}{section}

\titleformat{\title}
{\normalfont\scshape\center}{\title}{1em}{}
\titleformat{\section}
{\normalfont\scshape\center}{\thesection}{1em}{}
\titleformat{\subsection}
{\normalfont\bfseries}{\thesubsection}{1 em}{}
\usepackage{fancyhdr}
\usepackage[
backend=biber,
style=alphabetic,
sortlocale=de_DE,
natbib=true,
url=false, 
doi=true,
eprint=false
]{biblatex}
\begin{document}
		\maketitle
		\begin{abstract}
			We prove that a class of weakly partially hyperbolic endomorphisms on $\bbT^2$ are dynamically coherent and leaf conjugate to linear toral endomorphisms. Moreover, we give an example of a partially hyperbolic endomorphism on $\bbT^2$ which does not admit a centre foliation.\\
			
		\noindent \keyword{Partial hyperbolicity, Non-invertible dynamics, Dynamical coherence.}\\
		\noindent\codes{37C05, 37D30, 57R30.}\\
		
		\end{abstract}
	\section{Introduction}
	\funding{This work was partially funded by the Australian Research Council.}
	Partial hyperbolicity has been extensively studied as a mechanism for robust dynamical behaviour of diffeomorphisms. Weakly partially hyperbolic diffeomorphisms are particularly well understood in dimension 2, where a classification has been established. Compared to diffeomorphisms, the dynamics of non-invertible surface maps are less understood. In this paper, we study partially hyperbolic surface endomorphisms. We give a classification of a particular class of these maps up to leaf conjugacy. We also give an example from a different class which does not admit a centre foliation, showing that not all partially hyperbolic surface endomorphisms can be classified up to leaf conjugacy.
	
	We begin by recalling the definition of partial hyperbolicity in the invertible setting. Let $M$ be a compact connected manifold. A diffeomorphism $f:M\to M$ is \emph{(weakly) partially hyperbolic} if there exists a splitting of the tangent bundle
	\[ TM = E^c \oplus E^u \]
	which is invariant under $Df$, and 	such that for all $p\in M$ and all unit vectors $v^c \in E^c(p)$ and $v^u \in E^u(p)$,
	
	\begin{equation*}
	1 <\| Df^nv^u \| \quad \text{ and }\quad \|Df^nv^c\|<\|Df^nv^u\|. \label{ph}
	\end{equation*}

	In the non-invertible setting, it is more natural to define partial hyperbolicity in terms of cone families rather than an invariant splitting. A \emph{cone family} $\Cone \subset TM$ consists of a closed convex cone $\Cone(p) \subset T_p M$ at each point in $p \in M$. A cone family is \emph{$Df$-invariant} if $D_p f \left(\Cone(p)\right)$ is contained in the interior of $\Cone(f(p))$ for all $p\in M$. A map $f:M\to M$ is a \emph{(weakly) partially hyperbolic endomorphism} if it is a local diffeomorphism and it admits a cone family $\Cone^u$ which is $Df$-invariant and such that $1 < \| Df v^u \|$
	for all $v^u \in \Cone^u$. We call $\Cone^u$ an unstable cone-family. This definition of partial hyperbolicity coincides with that for diffeomorphisms when $f$ is invertible. In the non-invertible setting, admitting an unstable cone family is a robust property, whereas admitting an invariant splitting is not. Our focus is on surface endomorphisms, so let $M$ be two-dimensional. It can be shown that an unstable cone-family implies the existence of a centre direction, that is, a continuous $Df$-invariant line field $E^c\subset TM$ \cite[Section 2]{cropot2015lecture}.
	If we assume $M$ is orientable, the existence of $E^c$ implies that $M=\bbT^2$.

	A problem that arises when classifying diffeomorphisms is whether or not there exists a foliation tangent to the centre direction. Even a one-dimensional centre bundle does not necessarily integrate to a foliation. This is demonstrated in \cite{RHRHU-incoherent}, where they present a partially hyperbolic diffeomorphism on $\bbT^3$ which only admits an invariant centre \emph{branching foliation}, a collection of immersed surfaces which cover the manifold and do not pairwise topologically cross. In our setting, we say a partially hyperbolic endomorphism of $\bbT^2$ is \emph{dynamically coherent} if there exists an $f$-invariant foliation tangent to $E^c$.
	
	Suppose that $f,\, g:\bbT^2 \to \bbT^2$ are partially hyperbolic endomorphisms which are dynamically coherent. We say that  $f$ and $g$ are \emph{leaf conjugate} if there exists a homeomorphism $h: \bbT^2 \to \bbT^2$ which takes centre leaves of $f$ to centre leaves of $g$, and which satisfies
	\[
	h(f(\mathcal{L})) = g(h(\mathcal{L})).
	\]
	In \cite[Section 4A]{potrie}, it is proved that a partially hyperbolic diffeomorphism of $\bbT^2$ is dynamically coherent. Further, it is known that any partially hyperbolic diffeomorphism of $\bbT^2$ is leaf conjugate to a linear toral automorphism. This known classification is not precisely stated anywhere, but it is contained within Theorem \ref{coherence} of this paper.
	
	A partially hyperbolic endomorphism $f:\bbT^2\to\bbT^2$ induces a homomorphism of the fundamental group $\pi_1(\bbT^2) \cong \bbZ^2$. There is a unique linear toral endomorphism $A: \bbT^2 \to \bbT^2$ which induces the same homomorphism of $f$, and we call $A$ the \emph{linearisation} of $f$. One can show that $A$ is one of three types based on the eigenvalues $\lambda_1$ and $\lambda_2$ of the matrix inducing $A$. We give these types the following names:
	\begin{itemize}
		\item if $|\lambda_1| < 1 < |\lambda_2|$, we say $A$ is \emph{hyperbolic} if, 
		\item if $1<|\lambda_1| \leq |\lambda_2|$, we say $A$ \emph{expanding}, and
		\item if $1<|\lambda_1|$ and $|\lambda_2| = 1$, we say $A$ is \emph{non-hyperbolic}.
	\end{itemize}
	We will often relate the behaviour of $f$ to that of $A$. If $A$ is hyperbolic or expanding, then by the work of Franks \cite{Franks1}, $f$ is semiconjugated to $A$.
	
	The linearisation of a diffeomorphism is necessarily hyperbolic, and the known classification can be stated in terms of the linearisation: a partially hyperbolic diffeomorphism on $\bbT^2$ is dynamically coherent and leaf conjugate to its linearisation. We extend this classification to any partially hyperbolic endomorphism with hyperbolic linearisation.

	\begin{theorem}\label{coherence}
		Let $f: \bbT^2\to\bbT^2$ be a partially hyperbolic endomorphism, and suppose that its linearisation $A$ is hyperbolic. Then $f$ is dynamically coherent.	
	\end{theorem}
	
	\begin{theorem}\label{leafconj}
		Let $f: \bbT^2\to\bbT^2$ be a partially hyperbolic endomorphism, and suppose that its linearisation $A$ is hyperbolic. Then $f$ is leaf conjugate to $A$.	
	\end{theorem}
	
	We prove Theorems A and B by working on the universal cover $\bbR^2$. In this setting, some of the difficulties when working with non-invertible maps---such as the inability to take preimages or the use of an unstable direction---are alleviated.
	
	In \cref{coherenceconstruction} we prove \cref{coherence} using techniques similar to the invertible setting. We begin with an approximating foliation on the universal cover, and take large backward iterates under the lifted endomorphism. Using the hyperbolic linearisation, we show that the resulting limit is a centre foliation.
	
	We establish \cref{leafconj} in \cref{lcsection} by taking an average of the Franks semiconjugacy along centre leaves to obtain a leaf conjugacy. This averaging technique was first used by Fuller \cite{fuller1965}, and an outline of how this method can be used to establish a leaf conjugacy for partially hyperbolic diffeomorphisms is given in \cite[Section 8]{hp-survey}.
	
	The techniques used to prove Theorems \ref{coherence} and \ref{leafconj} are specific to endomorphisms with a hyperbolic linearisation. We show that in the case of a non-hyperbolic linearisation, incoherence is possible.
	\begin{theorem}\label{incoherence}
		There exists a partially hyperbolic endomorphism $f: \bbT^2 \to \bbT^2$ with non-hyperbolic linearisation that is dynamically incoherent.	
	\end{theorem}
	Our example is inspired by the partially hyperbolic diffeomorphism of $\bbT^3$ constructed in \cite{RHRHU-incoherent}, where branching of centre curves occurs on an invariant 2-torus tangent to $E^s \oplus E^c$. Similarly, our example has branching of centre curves occur on a circle tangent to $E^c$. We give the construction of the incoherent endomorphism in \cref{incosection}. \cref{incoherence} shows that not all partially hyperbolic endomorphisms on $\bbT^2$ can be classified up to leaf conjugacy. It also gives an example of a partially hyperbolic endomorphism on $\bbT^2$ which exhibits interesting behaviour similar to a 3-dimensional diffeomorphism, but which is not possible for a 2-dimensional diffeomorphism.

	\section{Dynamical coherence}
		\label{coherenceconstruction}
		In this section we prove \cref{coherence}. Let $f_0: \bbT^2 \to \bbT^2$ be a partially hyperbolic endomorphism and $A_0: \bbT^2 \to \bbT^2$ be its linearisation. Suppose that $A_0$ is hyperbolic in the sense defined in the introduction. Lift $f_0$ and $A_0$ to $f: \bbR^2 \to \bbR^2$ and $A: \bbR^2 \to \bbR^2$. As $f_0$ and $A_0$ induce the same homomophism on $\pi^1(\bbT^2)$, one can show that the lifted maps $f$ and $A$ are a finite distance apart. The map $f$ is a diffeomorphism which admits a splitting $E^c \oplus E^u$ \cite[Section 2]{manepughendo}. The unstable direction $E^c$ integrates to a foliation $\mathcal{F}^u$ of $\bbR^2$ by lines, and this foliation does not necessarily descend to $\bbT^2$ \cite[Section 2]{przytycki}. We will work on the universal cover to construct a centre foliation which descends to $\bbT^2$. Let $\mathcal{A}^s$ be the foliation of $\bbR^2$ whose leaves are the stable eigenlines of $A$, and similarly define $\mathcal{A}^u$. 
		
	 	Let $E^c_0$ be the centre direction of $f_0$ on $\bbT^2$. To construct a centre foliation, we first approximate $E^c_0$ by an $\eps$-close smooth distribution $E^{\eps}_0$ which remains transverse to the unstable cone family. Since $E^{\eps}_0$ is smooth, it integrates to a foliation of $\bbT^2$. We can lift this foliaton to a foliation $\mathcal{F}^{\eps}$ of $\bbR^2$. Any foliation lifted from $\bbT^2$ has leaves that are uniformly close to leaves of a foliation by lines (see Lemma 4.A.2 of \cite{potrie}). For $\mathcal{F}^{\eps}$, we can ensure this foliation of lines is not $\mathcal{A}^u$.
	 	
		\begin{lem}\label{notunstable}
			The distribution $E^{\eps}_0$ can be chosen so that the resulting foliation $\mathcal{F}^{\eps}$ has leaves at a finite distance from a line $L$ which is not in $\mathcal{A}^u$.	
		\end{lem}
		\begin{proof}
			By Pugh's $C^1$ closing lemma, up to a small perturbation, we can choose $E^{\eps}_0$ to have a periodic orbit. Then $\mathcal{F}^{\eps}$ must have a leaf which descends to a circle in $\bbT^2$. This leaf then lies a finite distance from line of rational slope. One can show that since $A$ hyperbolic, the lines of $\mathcal{A}^u$ have irrational slope, so $L \notin \mathcal{A}^u$. 
		\end{proof}
		
		The centre foliation is obtained by iterating $\mathcal{F}^{\eps}$ backward. For each $n$, define a foliation $\mathcal{F}_n$ by setting $\mathcal{F}_n(p) = f^{-n}(\mathcal{F}^{\eps}(f^n(p)))$. The sequence of leaves $\mathcal{F}_n(p)$ is a sequence of embedded copies of $\bbR$ through $p$ whose tangent lines $T\mathcal{F}_n(p)$ converge exponentially fast to the centre direction. An Arzela-Ascoli argument shows that $\mathcal{F}_n(p)$ has a convergent subsequence in the compact-open topology. Further, the limit of this subsequence is an embedding of $\bbR$ through $p$ tangent to $E^c$. Any sequence of leaves in $\mathcal{F}_n$ can have several subsequences which have a limit; we let $\mathcal{F}^c$ be the collection of all such limits of all leaves. It follows from this definition that the collection is $f$-invariant. We will show that curves in $\mathcal{F}^c$ are either disjoint or coincide, in turn implying that $\mathcal{F}^c$ is a foliation. Since $f$ is close to $A$ and each leaf of $\mathcal{F}^{\eps}$ is close to a line $L$, we consider how $A$ maps $L$. Let $L_n = A^{-n}(L)$, and observe that $L_n$ approaches the stable eigenline of $A$ for large $n$. We can use this to show that leaves of $\mathcal{F}_n$ are close to $L_n$. This allows us to establish that every leaf of $\mathcal{F}_n$ lies close to $L_n$.
		\begin{lem} \label{limitneighbourhood}
			There exists $R>0$ such that for all large $n$, each leaf of $\mathcal{F}_n$ lies in an $R$ neighbourhood of a translate of the line $L_n$.
		\end{lem}
		\begin{proof}
			If $X\subset\bbR^2$ is a subset, define the $R$-neighbourhood of $X$ by $U_R(X) = \{p\in\bbR^2: \mathrm{dist}(p,X)<R\}$. Through basic linear algebra, one can show there exists $0<\alpha<1$ and $N\in\bbN$ such that
			\[
			A^{-1}(U_{R}(L_n)) \subset U_{\alpha R}(L_{n+1})
			\]
			for all $n>N$ and any $R>0$. Consider $n>N$. If $K_0>0$ is the distance from $f$ to $A$, then $f^{-1}$ maps an $R$-neighbourhood of $L_n$ into an $R\alpha+K_0$ neighbourhood of $L_{n+1}$. Since $\alpha<1$, we can choose $R$ large enough so that $R>R\alpha+K_0$, so the result follows by induction.
		\end{proof}
		 Define $\pi^s: \bbR^2 \to \bbR$ as a linear map whose kernel is the unstable eigenline of $A$ and which maps the leaves of $\mathcal{A}^s$ onto $\bbR$. Similarly define $\pi^u: \bbR^2 \to \bbR$. As $L_n$ approaches $L$ for large $n$, the preceding lemma implies that every curve of $\mathcal{F}^c$ lies close to a line of $\mathcal{A}^s$. By invariance, all forward iterates of curves are close to lines of $\mathcal{A}^s$. This property can be stated in terms of $\pi^u$.
		\begin{lem}\label{closetostable}
			There is $C>0$ such that for any curve $\mathcal{L} \in \mathcal{F}^c$, the interval $\pi^u(f^n(\mathcal{L})) \subset \bbR$ has length at most $C$.
		\end{lem}
		Since $A$ is hyperbolic, we have the Franks semiconjugacy: the unique continuous surjective map $H: \bbR^2 \to \bbR^2$ that commutes with deck transformations and satisfies $A \circ H = H \circ f$ \cite{Franks1}. It follows that $H$ is a finite distance from the identity on $\bbR^2$.
		\begin{lem} \label{unstabletounstable}
			If $p\in\bbR^2$ and $q\in\mathcal{F}^u(p)$, then $H(q)\in\mathcal{A}^u(H(p))$.
		\end{lem}
		\begin{proof}
			As $H$ descends to a map on $\bbT^2$, it is uniformly continuous. If $q\in\mathcal{F}^u(p)$, then $|f^{-n}(p)-f^{-n}(q)|\to 0$, which implies
			\[
			|A^{-n}H(p)-A^{-n}H(q)|= |Hf^{-n}(p)-Hf^{-n}(q)| \to 0.
			\]
			Then we must have $H(q)\in\mathcal{A}^u(H(p))$.
		\end{proof}
		As $H$ is a finite distance from the identity, it then follows that leaves of $\mathcal{F}^u$ lie close to lines of $\mathcal{A}^u$.
		\begin{cor}\label{closetounstable}
			There is $D>0$ such that if $\mathcal{L}\in\mathcal{F}^u$, then the interval $\pi^s(\mathcal{L})$ has length at most $D$.
		\end{cor}
		Since the unstable and centre directions are transverse line fields on $\bbR^2$, we can say something about the structure between $\mathcal{F}^c$ and $\mathcal{F}^u$.
		\begin{prop}\label{uniquehit}
			If $\mathcal{L}^c \in \mathcal{F}^c$ and $\mathcal{L}^u\in\mathcal{F}^u$, then $\mathcal{L}^c$ cannot intersect $\mathcal{L}^u$ more than once.
		\end{prop}
		\begin{proof}
			Suppose that a curve in $\mathcal{F}^c$ intersects a leaf of $\mathcal{F}^u$ more than once. Since this curve of $\mathcal{F}^c$ is transverse to $\mathcal{F}^u$, then by the Poincar\'e-Bendixson Theorem, $\mathcal{F}^u$ must have a leaf which is a circle. As $\mathcal{F}^u$ is a foliation consisting of lines, this is a contradiction.
		\end{proof}
		
		In the setting of partially hyperbolic diffeomorphisms in dimension 3, if we lift a centre-stable branching foliation to the universal cover, it intersects an unstable leaf in at most one point. From this, one can prove a `length versus volume' inequality for neighbourhoods of unstable leaves. Such an argument is given for example in \cite[Lemma 3.3]{BBI1} and \cite[Lemma 5.5]{sol}. In our setting of endomorphisms in dimension 2, we show that similar results hold.
		
		\begin{lem}\label{prevol}
			 There exist constants $\eps,\,\eta$ with $1>\eps>\eta>0$ such that if $J^u$ is an unstable segment of length $1$ through a point $p$, then
			\begin{itemize}
				\item if $q\notin J^u$ is a point which satisfies $|p-q|<\eta$, then any complete centre curve through $p$ intersects $\mathcal{F}^u(q)$, and
				\item any complete centre curve with endpoints on $J^u \cap B_{\eta}(p)$ and $\mathcal{F}^u(q)$ is contained within $B_{\eps}(p)$.
			\end{itemize}
		\end{lem}
		\begin{proof}
			Descend to $\bbT^2$, where we have an unstable cone family $\Cone^u$ is transverse to $E^c$. Given a point $p$, we can find $\eta>0$ such that if $|p-q|>\eta$, then any curve tangent to $\Cone^u$ passing through $p$ intersects any sufficiently long centre segment through $q$. Moreover, if $\eta$ is small, this intersection is close to $p$ and $q$ with respect to distance along the curves. Compactness of $\bbT^2$ allows us to conclude the result.
		\end{proof}
		
		For an unstable segment $J$, recall that the unit neighbourhood of $J$ is defined by $U_1(J) = \{p\in\bbR^2: \mathrm{dist}(p,J)<1\}$. Then define $U_1^c(J)$ as the set of all $p \in U_1(J)$ such that there is a curve $\mathcal{L}\in\mathcal{F}^c$ passing through $p$ with $\mathcal{L} \cap J \neq \varnothing$.
		\begin{prop}\label{lengthvolume}
			There is $K>0$ such that if $J \subset \bbR^2$ is either an unstable segment or a centre segment, then
			\[
			\mathrm{volume}(U_1(J))> K \, \mathrm{length}(J).
			\] 
		\end{prop}
		\begin{proof}
			We prove this for an unstable segment, the argument for a centre segment is similar. By the preceding lemma, there is $\delta > 0$ such that $U_1^c(J)>\delta$ for any unstable segment $J$ of length greater than $1$. If $\mathrm{length}(J) > n$, there are disjoints subcurves $J_1,...,J_n \subset J$ of length one. \cref{uniquehit} states that a centre segment can intersect $J^u$ at most once, so the sets $U_1^c,...,U_n^c$ are pairwise disjoint. Then $
			\mathrm{volume}(U_1(J))>\delta n$, so the claim follows.
		\end{proof}
		The preceding proposition is a `length versus volume' inequality. It implies that unstable segments must be large in the $\pi^u$-direction.
		\begin{lem}\label{largeunstable}
			For any $M>0$ there is $\ell>0$ such that any unstable segment $J^u$ of length greater than $\ell$ is such that the interval $\pi^u(J^u)$ has length greater than $M$. In particular, $J^u$ must admit points $p$ and $q$ such that 
			\[
			|\pi^u(p)-\pi^u(q)|>M.
			\]
		\end{lem}
		\begin{proof}
			Let $\ell$ be a constant to be chosen later, and let $J^u$ be an unstable curve of length greater than $\ell$. Let $r$ be the length of the interval $\pi^u(f^n(J^u))$. \cref{closetounstable} states that $J^u$ is bounded in the $\pi^s$ direction, so
			\[
			\mathrm{vol}(U_1 f^n(J^u))\leq (r+2)(D+2).
			\]
			Then by \cref{lengthvolume},
			\[
			K \ell \leq (r+2)(D+2).
			\]
			Hence if $M>0$, we can choose $\ell$ sufficiently large so that $r>M$.
		\end{proof}
		We have shown that curves in $\mathcal{F}^c$ and $\mathcal{F}^u$ can intersect at most once, and that unstable leaves are unbounded in the $\pi^u$ direction. Both of these results can be strengthened.
		\begin{prop}\label{palmoff}
			The unstable foliation has the following properties:
			\begin{enumerate}%[(i)]
				\item \label{itm:intersection} Every leaf of $\mathcal{F}^u$ intersects every curve of $\mathcal{F}^c$ exactly once.
				\item \label{item:longendpoint}For all $M>0$, there is $\ell>0$ such that if $J^u$ is an unstable segment of length greater than $\ell$, then its endpoints $p$ and $q$ satisfy
				\[
				|\pi^u(p)-\pi^u(q)|>M.
				\]
			\end{enumerate}
		\end{prop}
		\begin{proof}
			To prove these properties, we adapt the proofs of 4.11, 4.12 and 4.13 in \cite{nil}. In fact, the proofs there apply to our current setting with the following modifications:
			\begin{itemize}
				\item the universal cover is $\bbR^2$ instead of $\mathcal{H}$,
				\item we replace the centre-stable through a point $p$, denoted $W^{cs}(p)$, with \emph{any} curve of $\mathcal{F}^c$ which passes through $p$,
				\item Proposition 3.1, Lemma 4.7 and Lemma 4.10 of \cite{nil} are replaced by \cref{uniquehit}, \cref{largeunstable}, \cref{closetostable} of the current paper, respectively. \qedhere
			\end{itemize}
		\end{proof}
		
		Now suppose we have two curves in $\mathcal{F}^c$ which intersect. As we iterate these curves forward, both curves remain close in the $\pi^u$ direction. This contradicts that an unstable segment between the curves will grow large in the $\pi^u$ direction.
		\begin{prop}
			Curves in $\mathcal{F}^c$ are either disjoint or coincide.
			\label{uniqbackward}
		\end{prop}
		\begin{proof}
			Suppose there exist two distinct curves $\mathcal{L}_1,\, \mathcal{L}_2 \in \mathcal{F}^c$ which pass through the same point $p\in\bbR^2$. Consider the intervals $\pi^uf^n(\mathcal{L}_1)$ and $\pi^uf^n(\mathcal{L}_2)$. Since both of these intervals contain $\pi^uf^n(p)$, their union is an interval, and \cref{closetostable} implies that the length of this interval is at most $2C$. In particular, if $q_1\in\mathcal{L}_1$ and $q_2\in\mathcal{L}_2$ are points connected by an unstable segment, then $|\pi^uf^n(q_1)-\pi^uf^n(q_2)|$ is bounded. This contradicts the second property of \cref{palmoff}.
		\end{proof}
	
		We can now conclude that $f$ is dynamically coherent.
		
		\begin{proof}[Proof of \cref{coherence}]
			By \cref{uniqbackward}, the collection $\mathcal{F}^c$ gives a partition of $\bbT^2$. These curves are tangent to $E^c$, so $\mathcal{F}^c$ is a foliation by Remark 1.10 in \cite{BoWi}. This foliation is tangent to the centre direction and $f$-invariant. One can show that foliation is invariant under deck transformations, so $\mathcal{F}^c$ descends to the desired foliation on $\bbT^2$.
		\end{proof}
		We say that two foliations $\mathcal{F},\,\mathcal{G}$ of $\bbR^2$ have global product structure if each leaf of $\mathcal{F}$ intersects every leaf of $\mathcal{G}$ precisely once.
		\begin{prop}
			The foliations $\mathcal{F}^c$ and $\mathcal{F}^u$ have global product structure.
		\end{prop}
		\begin{proof}
			 This follows immediately from the first claim of \cref{palmoff}.
		\end{proof}
		
	\section{Leaf conjugacy} \label{lcsection}
		Having shown dynamical coherence, we now construct a leaf conjugacy to prove Theorem \ref{leafconj}. We retain the notation from section \ref{coherenceconstruction}. Recall that $H$ is the Franks semiconjugacy, and define the associated functions $H^s, H^u: \bbR^2 \to \bbR$ by $H^s = \pi^s\circ H$ and $H^u = \pi^u\circ H$. The leaf conjugacy will be constructed by an averaging of the semiconjugacy $H$ along centre leaves. We first need to establish properties of the foliations $\mathcal{F}^u$, $\mathcal{F}^c$ and the semiconjugacy $H$.

		\begin{prop}
			The restriction of $H^u$ to a leaf of $\mathcal{F}^u$ is a homeomorphism to $\bbR$.
		\end{prop}
		\begin{proof}
			For $\mathcal{L}\in\mathcal{F}^u$, let $p$ and $q$ be points in $\mathcal{L}$. If $H^u(p)=H^u(q)$, then as $H$ is a semiconjugacy, $H^uf^n(q)=H^uf^n(p)$. Thus $|\pi^uf^n(p)-\pi^uf^n(q)|$ is bounded for all $n$, which is contradicted by \cref{largeunstable}. Hence $H^u$ is injective. \cref{unstabletounstable} implies that $H^s(\mathcal{L})$ is a point. By using this and the fact that $H$ is a finite distance from the identity, one can show that $H^u$ must be surjective.
		\end{proof}
		
		Now we can show that $H$ gives a bijection between leaves of $\mathcal{F}^c$ and those of $\mathcal{A}^s$.
		\begin{prop}\label{bijection}
			If $p \in \bbR^2$, then $q \in \mathcal{F}^c(p)$ if and only if $H(q) \in \mathcal{A}^s(H(p))$.
		\end{prop}
		\begin{proof}
			We first prove the forward implication. Suppose that there exist points $p, q \in \bbR^2$ such that $q \in \mathcal{F}^c(p)$ but $H(q) \notin \mathcal{A}^s(H(p))$. Since $H(p)$ and $H(q)$ lie in different lines of $\mathcal{A}^s$, then
			$\sup_n|\pi^u A^n(Hp) - \pi^u A^n(Hq)|= \infty.$
			As $H$ is a semiconjugacy which is a finite distance from the identity,
			\[
			\sup_n|\pi^u f^n(p) - \pi^u f^n(q)| = \infty.
			\]
			However $p$ and $q$ lie in the same centre leaf. By \cref{closetostable} this implies that
			\[
			\sup_n|\pi^u f^n(p) - \pi^u f^n(q)|<\infty,\] giving a contradiction.
			
		Now suppose that the converse does not hold. Then $H$ must map two distinct leaves of $\mathcal{F}^c$ to the same line in $\mathcal{A}^s$. Then by global product structure, there are points $p,\,q\in\bbR^2$ such that $q\in\mathcal{F}^u(p)$ but $H^u(q)=H^u(p)$. This contradicts that the restriction of $H^u$ to $\mathcal{F}^u(p)$ is a homeomorphism.
		\end{proof}
		
		Similar to how long unstable segments are large the $\pi^u$-direction, long centre segments must be large in the $\pi^s$-direction.
		\begin{lem}\label{increase}
			There is $T>0$ such that if $J^c$ is a centre leaf segment of length $T$ then the endpoints $p$ and $q$ satisfy
			\[\
			|\pi^s(p)-\pi^s(q)|>1.
			\]
		\end{lem}
		\begin{proof} 
			The centre foliation on $\bbR^2$ quotients down to a foliation on $\bbT^2$ which is equivalent to the suspension of a circle map with irrational rotation number.
			
%	The centre foliation on $\bbR^2$ quotients down to a foliation on $\bbT^2$ which is equivalent to a linear foliation consisting of lines of irrational slope.
%			We repeat the same steps that we followed in proving \cref{palmoff}, but now with the role of the unstable and centre directions swapped. Centre leaves are bounded in the $\pi^s$-direction by \cref{closetostable}, and the preliminary results of \cref{prevol}, \cref{lengthvolume}, and \cref{largeunstable} can be proved for centre segments by the same arguments. The only significant difference between $\mathcal{F}^c$ and $\mathcal{F}^u$ is that $f$ uniformly expands unstable segments, however this fact is not used in proving any of the listed propositions.
		\end{proof}

		We now use the averaging method as described in \cite[Section 8]{hp-survey} to create a leaf conjugacy $h$ on $\bbR^2$. We begin by defining $h$ on a given centre leaf $\mathcal{L} \in \mathcal{F}^c$. Let $\alpha: \bbR\to\bbR^2$ be an arc length parametrisation of $\mathcal{L}$. For $p \in \mathcal{L}$, let $s = \alpha^{-1}(p)$. Let $T$ be as in \cref{increase} and define $h(p)$ as the unique point in $H(\mathcal{L})$ which satisfies
		\[
		\pi^s h(p) = \frac{1}{T}\int_0^T \pi^s \alpha(s+t)\, dt.
		\]
		
		One can show that $h(p)$ is independent of the choice of parametrisation.
		
		\begin{lem}\label{diffeo}
			For any $\mathcal{L}\in\mathcal{F}^c$, the map $h: \mathcal{L} \to H(\mathcal{L})$ is a $C^1$ diffeomorphism. 	
		\end{lem}
		\begin{proof}
			By the fundamental theorem of calculus we have
			$$
			\frac{d}{dt}(\pi^s h (\alpha(t_0)))
			= \frac{1}{T}(\pi^s \alpha(t_0 + T) - \pi^s \alpha(t_0))\\
			> \frac{1}{T},
			$$
			where the final inequality uses \cref{increase}. The result follows.
		\end{proof}
		
		By defining $h$ on every leaf of $\mathcal{F}^c$, we obtain a map $h:\bbR^2 \to \bbR^2$. This map is the desired leaf conjugacy.
		
		\begin{proof}[Proof of \cref{leafconj}]
			By our construction,  $h:\bbR^2 \to \bbR^2$ is a bijection which takes leaves of $\mathcal{F}^c$ to leaves of $\mathcal{A}^s$ and satisfies $h(f(\mathcal{L})) = A(h(\mathcal{L}))$ for $\mathcal{L}\in\mathcal{F}^c$. By \cref{diffeo}, $h$ is continuous when restricted to a leaf of $\mathcal{F}^c$, and by continuity of the semiconjugacy $H$ and the foliation $\mathcal{F}^c$, $h$ is continuous on all of $\bbR^2$. Since $H$ is surjective, as is $h$, so by the invariance of domain theorem, $h$ is a homeomorphism. Hence $h$ is a leaf conjugacy from $f$ to $A$ on $\bbR^2$. One can show that $h$ descends to the desired leaf conjugacy on $\bbT^2$.
			
			%For paranoia's sake, we also show that $h$ descends. Consider $\tau \in \pi_1(\bbR^2)$ and identify $\tau$ with $(n,m) \in \bbZ^2$. Let $\alpha$ be an arclength parametrisation of $\tau L$, and observe that if $\beta$ is an arclength parametrisation of $L$, then $\tau \beta$ is an arclength parametrisation of $\tau L$, so that
			%\begin{align*}
			%\pi^s h(\tau p)
			%&= \frac{1}{T}\int_0^T \pi^s \alpha(\alpha^{-1}\tau p+t)\, dt\\
			%&= \frac{1}{T}\int_0^T \pi^s \tau\beta(\beta^{-1}\tau^{-1}\tau p+t)\, dt \\
			%&= \frac{1}{T}\int_0^T \pi^s(\beta(\beta^{-1}(p)+t) + (n,m))\,dt\\
			%&= \pi^s(n,m) + \pi^sh(p)\\
			%& = \pi^s \tau h( p)
			%\end{align*}
			%So $h$ commutes with deck transformations and thus descends to a leaf conjugacy $h_0: \bbT^2 \to \bbT^2$.
		\end{proof}
	\section{Incoherent example} \label{incosection}
		 When the linearisation of $f$ is not hyperbolic, then $f$ is not necessarily dynamically coherent. We construct an example of a partially hyperbolic endomorphism with non-hyperbolic linearisation which does not admit a centre foliation, proving Theorem \ref{incoherence}. This construction is inspired by the incoherent partially hyperbolic diffeomorphism on $\bbT^3$ in \cite{RHRHU-incoherent}.
		 Consider a Morse-Smale diffeomorphism $\Psi: \bbS^1 \to \bbS^1$ with fixed points at $0$ and $\frac{1}{2}$ such that $
		 \Psi'(0) <\frac{1}{2}$ and $\Psi'(\frac{1}{2}) > 2.$ Define $g:\bbS^1 \times\bbS^1 \to \bbS^1 \times \bbS^1$ by $g(x,y) = (2 x,\Psi(y))$. Observe that the expansion of $\Psi$ is stronger than the doubling map on the circle at $\frac{1}{2}$ but weaker at $0$.
		 The incoherent example is an explicit deformation of $g$. Define $f:\bbT^2\to\bbT^2$ by
		 \[
		 f(x,y) = (2 x + \cos(2\pi y)+1, \Psi(y)).
		 \]
		 
		 We will show that $f$ is partially hyperbolic by showing that it admits an invariant partially hyperbolic splitting $E^u \oplus E^c$. We seek a semiconjugacy $h: \bbT^2 \to \bbS^1$ so that $h \circ f = 2 h$. The preimage of points in $\bbS^1$ under $h$ will be centre curves of $f$. If we seek $h$ in the form $h(x,y) = x-u(y)$ for some function $u: \bbS^1 \to \bbR$, then $u$ must satisfy
		 \begin{equation}
		 u(\Psi(y)) - 2 u(y) = \cos(2\pi y)+1. \label{tce}
		 \end{equation}
		 \begin{lem}
		 	The equation (\ref{tce}) admits solutions $\beta,\, \gamma: \bbS^1 \to \bbR$ given by
		 	\begin{align*}
		 	\beta(y) = \frac{1}{2}&\sum_{k=1}^{\infty}2^k\left( \cos(2\pi\Psi^{-k}(y))+1\right),\\
		 	\gamma(y) = -\frac{1}{2}&\sum_{k=0}^{\infty}2^{-k} \left( \cos(2\pi\Psi^{k}(y))+1\right).
		 	\end{align*}
		 	The function $\beta$ is not well-defined at $0$, but is well defined and $C^1$ on $\bbS^1\setminus\{0\}$. The function $\gamma$ is well-defined and continuous on all of $\bbS^1$ and is $C^1$ on $\bbS^1\setminus\{1/2\}$. 
		 	\label{cohom}
		 \end{lem}
		 
		 	The preceding lemma can be proved by bounding the series and their termwise derivatives by geometric series using the conditions on $\Psi'$ at $0$ and $1/2$.

		 We now define subbundles $E^c, E^u \subset T\bbT^2$. Let
		 \begin{itemize}
		 \item $E^u(x,y)$ be spanned by $(\beta'(y),1)$ for $y\neq 0$, and by $(1,0)$ for $y=0$,
		 \\
		 \item $E^c(x,y)$ be spanned by $(\gamma'(y),1)$ for $y\neq \frac{1}{2}$, and  by $(1,0)$ for $y=\frac{1}{2}$.
		 \end{itemize}
		 We will show that these define an invariant partially hyperbolic splitting.
		 \begin{lem}
		 	The bundles $E^u$ and $E^c$ are continuous.
		 \end{lem}
		 \begin{proof}
		 	By \cref{cohom}, the bundles  $E^c$ and $E^u$ are continuous on $S$ except for the circles $y= \pi$ and $y =0$ respectively. In \cite{RHRHU-incoherent} (see the proof of Lemma 2.5(1)), it is proved that
		 	\[
		 	\lim_{y \to \pi}|\beta'(y)| =\lim_{y \to 0}|\gamma'(y)| =\infty.
		 	\]
		 	This ensures that both $E^c$ and $E^u$ are continuous at $\frac{1}{2}$ and $0$ respectively.
		 \end{proof}
		 \begin{lem}
				 The bundles $E^u$ and $E^c$ are transverse.
		 \end{lem}
		 \begin{proof}
		 	Since it is clear that $E^u$ and $E^c$ do not coincide on the invariant circles $y=0$ and $y=\frac{1}{2}$, it suffices to show that $\gamma'$ and $\beta'$ are not equal on $\bbS\setminus\{0,\frac{1}{2}\}$. From the definition of $\beta$ and $\gamma$, we see that $\beta'>0>\gamma'$ on $(0,\frac{1}{2})$ and $\beta'<0<\gamma'$ on $(\frac{1}{2},1)$, hence $\gamma'$ and $\beta'$ are not on $\bbS\setminus\{0,\frac{1}{2}\}$.
		 	\begin{comment}
		 	Since it is clear that $E^u$ and $E^c$ do not coincide on the invariant circles $y=0$ and $y=\frac{1}{2}$, it suffices to show that $\gamma'$ and $\beta'$ are not equal on $\bbS\setminus\{0,\frac{1}{2}\}$. By continuity of the bundles, $\beta'$ and $\gamma'$ are not equal on a neighbourhood of $\frac{1}{2}$. This neighbourhood contains some forward image of every point under $\Psi$, and this fact can be used to show that they are not equal anywhere. For details, see \cite{RHRHU-incoherent}.
		 	\end{comment}
		 \end{proof}
	
		 \begin{lem}
		 	The splitting $T\bbT^2 = E^u \oplus E^c$ is Df-invariant.
		 \end{lem}
		 \begin{proof}
		 	First observe that the circle $y=\frac{1}{2}$ is an invariant circle tangent to $E^c$. To see that $E^c$ is invariant elsewhere, let $p = (x,y)\in\bbT^2\setminus \{y=\frac{1}{2}\}$ and observe that derivative of $f$ is given by
		 	\[
		 	D_pf =
		 	\left(
		 	\begin{matrix}
		 	2  &-2\pi \sin(2\pi y)\\
		 	0 &\Psi'(y)
		 	\end{matrix}
		 	\right)
		 	.
		 	\]
		 	By differentiating the equation \eqref{tce}, it is straightforward to show that $E^u$ and $E^c$ are invariant.
		 \end{proof}
		 
		An outline of proof for the following result is given for example in \cite[Proposition 15.1]{ergcomp} for the invertible setting, the reasoning for our setting follows identically.
		 \begin{lem} \label{NW}
		 	Let $f$ be a partially hyperbolic endomorphism of $\bbT^2$ with an invariant splitting $TM = E^c \oplus E^u$. If for every point $p$ in the non-wandering set we have
		 	\begin{equation}
		 	\| D_pf\vert_{E^c}| < \| D_pf\vert_{E^u}\|, \quad 1<\|D_pf \vert_{E^u}\|, \label{NWe}
		 	\end{equation}
		 	then the above inequalities hold for all $p\in M$.
		 \end{lem}

		 \begin{prop}
		 	The endomorphism $f:\bbT^2\to\bbT^2$ is partially hyperbolic.
		 \end{prop}
		 \begin{proof}
		 	The non-wandering set of $f$ is the union of the invariant circles $y = 0$ and $y =1/2$. Since we can compute the derivative of $f$, it is straightforward to show that $E^c \oplus E^u$ satisfies inequalities (\ref{NWe}) on these circles. By \cref{NW}, these inequalities hold on $\bbT^2$, so $f$ is partially hyperbolic.
		 \end{proof}
		 
		 We can now prove our final result.
		 \begin{proof}[Proof of \cref{incoherence}]
		 	We will compute all the integral curves of $E^c$. For $p \in\bbT^2$, define a continuous curve $\sigma_p : [0,1] \to \bbT^2 $ by
		 	\begin{equation*}
		 	\sigma_p(t) = p + (\gamma(t),t).
		 	\end{equation*}
		 	The collection of curves $\sigma_p$ for each $p$ are $C^1$ and tangent to $E^c$ on the region $\bbT^2 \setminus \{y=\frac{1}{2}\}$. Since $E^c$ is $C^1$ on this domain, the collection is precisely the integral curves of $E^c$ on $\bbT^2 \setminus \{y=\frac{1}{2}\}$. The invariant circle $y=\frac{1}{2}$ is the remaining integral curve of $E^c$. Since $\gamma'$ is negative on $(0,\frac{1}{2})$ and positive on $(\frac{1}{2},1)$, the integral curves of $E^c$ then form a branching foliation with curves that coincide on $y=\frac{1}{2}$, as depicted in \cref{branching}. This implies a foliation tangent to the centre direction cannot exist.
		 \end{proof}
	 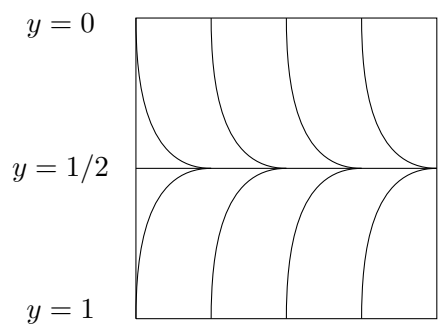
\begin{figure} 
	 	\centering
	 	\begin{tikzpicture}
	 	
	 	\draw (0,0)  -- (4,0) -- (4,4) -- (0,4) -- (0,0);
	 	\draw (0,0) to[out=90,in=180] (1,2);
	 	\draw (0,4) to[out=-90,in=180] (1,2);
	 	\draw (1,0) to[out=90,in=180] (2,2);
	 	\draw (1,4) to[out=-90,in=180] (2,2);
	 	\draw (2,0) to[out=90,in=180] (3,2);
	 	\draw (2,4) to[out=-90,in=180] (3,2);
	 	\draw (3,0) to[out=90,in=180] (4,2);
	 	\draw (3,4) to[out=-90,in=180] (4,2);
	 	\draw (0,2)  -- (4,2);
	 	\draw (-1,0.1) node {$y=1$};
	 	\draw (-1,2) node {$y=1/2$};
	 	\draw (-1,3.9) node {$y=0$};
	 	
	 	\end{tikzpicture}
	 	\caption{The centre branching foliation.}
	 	\label{branching}
	 \end{figure}
 
 	\acknowledgement 
 	
	\printbibliography
	
\end{document}